\documentclass{amsart}
\usepackage{ms}
\usepackage[usenames,dvipsnames]{xcolor}
\usepackage{xspace}
\usepackage{accents}
\usepackage{dirtytalk}
\usepackage[normalem]{ulem}

\usepackage{lineno}

\theoremstyle{plain}

\newcommand{\R}{\mathbb{R}}

\newcommand{\ignore}[1]{}

\begin{document}
\title[Second-order Morley's without Large Cardinals]{The Second-order Version of Morley's Theorem on the Number of Countable Models does not Require Large Cardinals}

\author[F. D. Tall]{Franklin D. Tall${}^1$}
\address[F. D. Tall]{University of Toronto, Department of Mathematics, 40 St. George St., Toronto, Ontario, Canada M5S 2E4}
\email{f.tall@utoronto.ca}
\urladdr{http://www.math.toronto.edu/tall/}
\thanks{${}^1$ Supported by Natural Sciences and Engineering Research Council Grants RGPIN-2016-06319 and RGPIN-2023-03420}

\author[J. Zhang]{Jing Zhang${}^2$}
\address[J. Zhang]{University of Toronto, Department of Mathematics, 40 St. George St., Toronto, Ontario, Canada M5S 2E4}
\email{jingzhan@alumni.cmu.edu}
\urladdr{}
\thanks{${}^2$ Supported by Natural Sciences and Engineering Research Council discovery grants}
\date{\today}

\subjclass[2020]{Primary 03C85, 03C55, 03E35, 03C52}

\keywords{Morley's theorem, countable models, Cohen forcing, $\sigma$-projective equivalence relations, large cardinals, generic absoluteness.}

\begin{abstract}
    The consistency of a second-order version of Morley’s Theorem on the number of countable models was proved in [EHMT23] with the aid of large cardinals. We here dispense with them.
\end{abstract}
\maketitle
\section{Introduction}

\emph{Vaught's Conjecture} \cite{Vaught}, which asserts that a countable first-order theory must have either at most countably many or exactly $2^{\aleph_0}$ many non-isomorphic countable models, is one of the most important problems in Model Theory. A strong positive result about Vaught's Conjecture is a result of the late Michael Morley \cite{Morley1970} which states that the number of isomorphism classes of countable models of a countable first-order theory is always at most $\aleph_1$ or exactly $2^{\aleph_0}$. Under this formulation, the result follows trivially from the continuum hypothesis. To avoid this artifact, one can identify countable models with members of the Cantor set (see \cite{Gao} or \cite{Eagle2023}) and prove: 

\begin{thm}[Absolute Morley]
    Let $T$ be a first-order theory (or more generally, a sentence of $L_{\omega_1, \omega}$) in a countable signature. Then either $T$ has at most $\aleph_1$ isomorphism classes of countable models, or there is a perfect set of non-isomorphic countable models of $T$.    
\end{thm}

The isomorphism relation among countable models can be formulated as a $\mathbf{\Sigma}_1^1$
equivalence relation, using a code for $T$ as a parameter. It is then easy to see that the following result \cite{Burgess} (see \cite{Gao}) is a strengthening of the Absolute Morley Theorem:

\begin{thm}
    Let $E$ be a $\mathbf{\Sigma}_1^1$ equivalence relation on $\R$. If there is no perfect set of pairwise inequivalent reals, then there are at most $\aleph_1$ equivalence classes.
\end{thm}

\emph{Second-order logic} is the natural generalization of first-order logic to a two-sorted language with variables for relations as well as for individuals. For a precise formulation of its syntax and semantics, see e.g. \cite{Eagle2023}. We then can formulate:

\emph{Second-order Absolute Morley}: If $T$ is a second-order theory in a countable signature, then either $T$ has at most $\aleph_1$ isomorphism classes of countable models, or there is a perfect set of (in particular $2^{\aleph_0}$ many) non-isomorphic countable models of $T$.

The \emph{$\sigma$-projective} hierarchy is obtained by extending the projective hierarchy up through the countable ordinals. See \cite{Kechris} and \cite{AMS21}. We can then formulate:

\emph{Second-order Absolute Morley for $\sigma$-projective equivalence relations} Let $E$ be a $\sigma$-projective equivalence relation on $\R$. If there is no perfect set of pairwise $E$-inequivalent reals, then $E$ has at most $\aleph_1$ equivalence classes.

Note that \emph{Second-order Absolute Morley for $\sigma$-projective equivalence relations} implies \emph{Second-order Absolute Morley}. This is done in \cite[Section 2.3]{Eagle2023} by showing that the complexity of the class of reals that code countable models of a given second-order theory is  a countable intersection of projective sets, in particular, is $\sigma$-projective.

This motivates the following apparently slightly weaker assertion:

\emph{Second-order Absolute Morley for countable intersections of projective sets} Let $E$ be an equivalence relation on $\R$ which is a countable intersection of projective sets. If there is no perfect set of pairwise $E$-inequivalent reals, then $E$ has at most $\aleph_1$ equivalence classes.

We \emph{do not know} if \emph{Second-order Absolute Morley for countable intersections of projective sets} is strictly weaker than \emph{Second-order Absolute Morley for $\sigma$-projective equivalence relations}.

In \cite{Eagle2023}, the following results are established:

\begin{thm*}[Theorem A]
    Force over $L$ by first adding $\aleph_2$ Cohen reals and then $\aleph_3$ random reals. In the resulting universe of set theory, $2^{\aleph_0} = \aleph_3$ but there is a second-order theory $T$ in a countable signature such that the number of non-isomorphic models of $T$ is exactly $\aleph_2$.
\end{thm*}

\begin{thm*}[Theorem C]
    If there are infinitely many Woodin cardinals, then there is a model of set theory in which Second-order Absolute Morley for countable intersections of projective sets holds.
\end{thm*}

The authors of \cite{Eagle2023} state as their first problem:

\textit{Prove that large cardinals are necessary to prove the consistency of Second-order Absolute Morley.}

We shall refute that conjecture here by proving:

\begin{thm}\label{thm:refute}
    Adjoin at least $\aleph_2$ Cohen reals to a model of $\CH$. Then Second-order Absolute Morley for $\sigma$-projective equivalence relations holds in the resulting model.
\end{thm}

The main idea of both \ref{thm:refute} and Theorem C occurs in the earlier work \cite{FM95}, in which Foreman and Magidor prove:

\begin{thm*}[Theorem B]
    In the usual iterated forcing model of $\PFA$ (thus assuming the existence of a supercompact cardinal), if $E$ is an equivalence relation on $\R$ such that $E$ is a member of $L(\R)$, then $E$ has either no more than $\aleph_1$ equivalence classes or else perfectly many equivalence classes.
\end{thm*}

The key is \emph{generic absoluteness}. Call an equivalence relation on sets of reals \emph{thin} if it does not have a perfect set of equivalence classes. Simplifying the argument by considering $E$'s definable from a real $r$, we note they prove (using the large cardinal) that the formula $\phi$ that defines $E$ defines a thin equivalence relation $E'$ in cofinally many intermediate models that contain $r$ and satisfy $\mathrm{CH}$. This is \emph{downwards generic absoluteness}. In any of such intermediate models, since $\CH$ holds, $E'$ has $\leq\aleph_1$ equivalence classes. Next they show that the rest of the forcing (after which the formula $\phi$ defines $E$) cannot add a new equivalence class to $E'$ unless $E$ has perfectly many equivalence classes in the final model. This is \emph{upwards generic absoluteness}. \cite{Eagle2023} follows the same approach but with a weaker large cardinal hypothesis. We follow the same approach here, but it turns out that because Cohen real forcing is so simple and homogeneous and because adding one Cohen real by forcing adds perfectly many, we don't need the large cardinal. This latter observation substitutes for upwards generic absoluteness, and hence no large cardinal is needed for that. If we add $\aleph_2$ many Cohen reals, then the real that codes the $\sigma$-projective set appears at an initial stage at which $\CH$ holds. If we add more than $\aleph_2$ Cohen reals, we need to apply an automorphism argument to get that without loss of generality, we may assume $r$ appears in the first $\omega_1$ stages. The required downwards generic absoluteness is proved by induction on the complexity of the $\sigma$-projective formulas that define our equivalence relations. A version of our generic absoluteness theorem was proved by Joan Bagaria many years ago with essentially the same proof, but he never published it. We rediscovered it and are not aware of any published reference. Now for the details.

Let us start with the definition of $\sigma$-projective formulas and $\sigma$-projective sets of reals. For the semantics of the infinitary logic $\mathcal{L}_{\omega_1,\omega}$, we refer the reader to \cite{Marker}.

\begin{definition}
Let $\phi$ be a formula in the language $\mathcal{L}_{\omega_1, \omega}$. 
\begin{enumerate}
\item $\phi$ is $\Sigma^1_0(\vec{x})$ if it is a disjunction of quantifier-free formulas with no infinite connectives with finitely many free variables $\vec{x}=(x_0,\cdots, x_{k-1})$ where $k\in \omega$.
\item $\phi$ is $\Pi^1_\alpha(\vec{x})$ if it is $\mathcal{L}_{\omega_1, \omega}$-equivalent to the negation of a $\Sigma^1_{\alpha}(\vec{x})$ formula.
\item $\phi$ is $\Sigma^1_{\alpha+1}(\vec{x})$ if it is of the form $\exists y \psi(y,\vec{x})$ where $\psi(y,\vec{x})$ is a $\Sigma^1_{\alpha}(y,\vec{x})$ formula.
\item $\phi$ is $\Sigma^1_{\alpha}(\vec{x})$ where $\alpha<\omega_1$ is a limit ordinal, if it is of the form $\bigvee_{k\in \omega} \phi_k$ where each $\phi_k$ is a $\Sigma^1_{\alpha_k}(\vec{x})$ formula for some $\alpha_k<\alpha$.
\end{enumerate}
\end{definition}

We say a formula $\phi$ in the language $\mathcal{L}_{\omega_1, \omega}$ is \emph{$\sigma$-projective} if it is $\Sigma^1_{\alpha}(\vec{x})$ for some $\alpha<\omega_1$. Recall that $H(\omega_1)$ is the collection of all hereditarily countable sets.

\begin{definition}
A set of reals $A$ is $\Sigma^1_\alpha(\vec{y})$ for $\alpha<\omega_1$ and reals $\vec{y}$, if there exists a $\Sigma^1_\alpha(\vec{x})$ formula $\phi(z, \vec{x})$ such that $A=\{r\in \mathbb{R}: H(\omega_1)\models_{\mathcal{L}_{\omega_1,\omega}}\phi(r,\vec{y})\}$.
\end{definition}

We say $A$ is \emph{$\sigma$-projective} if there is an $\alpha<\omega_1$ and a finite tuple of reals $\vec{y}$ such that $A$ is $\Sigma^1_\alpha(\vec{y})$. It is not hard to see that the collection of $\sigma$-projective sets is exactly the smallest $\sigma$-algebra on $\mathbb{R}$ containing the open sets and closed under continuous images. We refer the reader to \cite{AMS21} and \cite{Aguilera} for more information. In what follows, for cardinals $\alpha,\beta$, $Add(\alpha, \beta)$ is the standard Cohen poset for adding $\beta$ many Cohen subsets of $\alpha$. Namely, conditions are partial functions from $\beta \to 2$ with domain of size $<\alpha$, ordered by inclusion.

\begin{lemma}\label{lemma: 2step}
For any cardinals $\lambda,\kappa$, the following is true in $V^{Add(\omega,\omega_1)\times Add(\omega,\lambda)}$: for any $\sigma$-projective formula $\varphi(\bar{x})$ and any $\bar{a}\in \mathbb{R}^{|\bar{x}|}$
$$H(\omega_1)\models_{\mathcal{L}_{\omega_1, \omega}}\varphi(\bar{a}) \text{ if and only if } \Vdash_{Add(\omega,\kappa)} ``H(\omega_1)\models_{\mathcal{L}_{\omega_1,\omega}} \varphi(\bar{a})".$$
\end{lemma}

\begin{remark}
Since all the $\mathcal{L}_{\omega_1, \omega}$-formulas we consider use real parameters and all the quantifiers are bounded by the set of real numbers, they are absolute between $V$ and $H(\omega_1)$, since the latter contains all reals and is closed under countable sequences.
\end{remark}

\begin{proof}
We induct on the complexity of the formulas. Observe that the homogeneity of Cohen forcing implies that if for some $p$, $p\Vdash_{Add(\omega,\kappa)} H(\omega_1)\models_{\mathcal{L}_{\omega_1,\omega}} \varphi(\bar{a})$, then $\Vdash_{Add(\omega,\kappa)} H(\omega_1)\models_{\mathcal{L}_{\omega_1,\omega}} \varphi(\bar{a})$. In particular, this means if we have proved the result for $\mathbf{\Sigma}^1_\xi$ formulas, then we get the result for $\mathbf{\Pi}^1_\xi$ formulas immediately.

Suppose we have proved the theorem for all formulas that are $\mathbf{\Sigma}^1_\nu$ for $\nu<\xi<\omega_1$ and $\xi$ is a limit, then $\varphi(\bar{x})$ is of the form $\bigvee_n \{\psi_n(\bar{x}): \psi_n\in \mathbf{\Sigma}^1_{\xi_n}\}$ for a sequence $\langle \xi_n<\xi: n\in \omega\rangle$. Let $G\times H\subseteq Add(\omega,\omega_1)\times Add(\omega,\lambda)$. In $V[G\times H]$, suppose $H(\omega_1)\models_{\mathcal{L}_{\omega_1, \omega}}\varphi(\bar{a})$, equivalently, there is an $n\in \omega$, $H(\omega_1)\models_{\mathcal{L}_{\omega_1, \omega}} \psi_n(\bar{a})$. By the induction hypothesis, we know that $\Vdash_{Add(\omega,\kappa)} H(\omega_1)\models_{\mathcal{L}_{\omega_1, \omega}} \psi_n(\bar{a})$. As a result, $\Vdash_{Add(\omega,\kappa)} H(\omega_1)\models_{\mathcal{L}_{\omega_1, \omega}} \varphi(\bar{a})$. A similar argument shows that if $H(\omega_1)\models_{\mathcal{L}_{\omega_1, \omega}} \neg \varphi(\bar{a})$, then $\Vdash_{Add(\omega,\kappa)}H(\omega_1)\models_{\mathcal{L}_{\omega_1, \omega}} \neg \varphi(\bar{a})$.

Suppose we have proved the theorem for $\mathbf{\Sigma}_{\xi}^1$-formulas for $\xi<\omega_1$. Let $\exists x \psi(x, \bar{y})$ be a $\mathbf{\Sigma}_{\xi+1}^1$-formula. Suppose $$V[G\times H]^{Add(\omega,\kappa)}\models ``H(\omega_1)\models_{\mathcal{L}_{\omega_1, \omega}}\varphi(\bar{a})",$$ where $\bar{a}\in V[G\times H]$. Given $(p,q)\in P=_{def} Add(\omega,\omega_1)\times Add(\omega,\lambda)$ and $\dot{x}$ a $P\times Add(\omega,\kappa)$-name and $\dot{\bar{a}},\dot{\varphi}$ $P$-names such that $(p,q)$ forces the above holds for the respective names, find $\alpha<\omega_1, A\subseteq \lambda,B\subseteq \kappa$ countable such that these conditions and names are in $Add(\omega,\alpha)\times Add(\omega, A)\times Add(\omega, B)$ or are $Add(\omega,\alpha)\times Add(\omega, A)\times Add(\omega, B)$-names.

Define an automorphism $\pi$ on $P\times Add(\omega, \kappa)$ as follows: $\pi(r,s,t)=(r^*,s^*, t^*)$ if and only if
\begin{enumerate}
\item $r\restriction \omega_1 - [\alpha, \alpha+otp(B)) = r^*\restriction \omega_1 - [\alpha, \alpha+otp(B))$, $s= s^*$,
\item $r^*(\alpha+i)=t(i)$, $t^*(i)=r(\alpha+i)$ for all $i\in B$,
\item $t^*\restriction \kappa-B = t\restriction \kappa-B$.
\end{enumerate}

In particular, $\pi$ fixes $(p,q)$, $\dot{\bar{a}}$ and $\dot{\varphi}$. Therefore, $(p,q, \emptyset)\Vdash_{P\times Add(\omega,\kappa)} \psi(\pi(\dot{x}), \dot{\bar{a}})$. 

Let $G\times H\times R \subseteq P\times Add(\omega,\kappa)$ be generic containing $(p,q ,\emptyset)$. Then we know that $V[G\times H\times R]\models H(\omega_1)\models_{\mathcal{L}_{\omega_1,\omega}}\psi((\pi(\dot{x})^{G\times H\times R}), \bar{a})$. Let $x^*  =\pi(\dot{x})^{G\times H\times R}$. By the definition of $\pi$, we know that $\pi(\dot{x})$ is an $Add(\omega, \omega_1)\times Add(\omega,\lambda)$-name, in particular, $x^*\in V[G\times H]$. By the induction hypothesis, we know that $V[G\times H\times R]\models H(\omega_1)\models_{\mathcal{L}_{\omega_1,\omega}} \psi(x^*, \bar{a})$ if and only if $V[G\times H]\models H(\omega_1)\models_{\mathcal{L}_{\omega_1,\omega}} \psi(x^*, \bar{a})$. Therefore, $V[G\times H]\models H(\omega_1)\models_{\mathcal{L}_{\omega_1,\omega}}\exists x \psi(x, \bar{a})$, as desired. \end{proof}
\begin{remark}
The result holds in a more general context, namely, after adding $\aleph_1$ many Cohen reals, there exists an elementary embedding $j: (L(\mathbb{R}))^V \to (L(\mathbb{R}))^{V[G]}$ that is identity on the ordinals where $G$ is any further Cohen extension. In particular, all $\mathcal{L}_{\omega_1,\omega}$ sentences with reals and ordinals as parameters are absolute between $L(\mathbb{R})$ and $(L(\mathbb{R}))^{V[G]}$.
This was known to Bagaria and probably others. The proof we supply is only for what we need.
\end{remark}

\begin{lemma}\label{lemma: mutual}
For any $\lambda\geq \omega_1$, the following holds in $V^{Add(\omega, \lambda)}$: 
let $E$ be a $\sigma$-projective equivalence relation; if $\Vdash_{Add(\omega,1)} ``\exists \sigma\in \mathbb{R}$ such that for all $r\in V$, $\neg (\sigma E r)$", then $\Vdash_{Add(\omega,1)}$ there exist perfectly many $E$-classes.
\end{lemma}

\begin{proof}

Let $\dot{\sigma}$ be an $Add(\omega,1)$-name such that $\Vdash_{Add(\omega,1)} ``\neg \dot{\sigma} E r$ for any $r\in V$".

	\begin{claim}\label{claim:notEequiv}
	$\Vdash_{Add(\omega,1)\times Add(\omega,1)}\neg \dot{\sigma}_{\mathrm{left}} E \dot{\sigma}_{\mathrm{right}}$, where $\dot{\sigma}_{\mathrm{left}}$ ($\dot{\sigma}_{\mathrm{right}}$) is the name $\dot{\sigma}$ produced from the left (right) generic.
	\end{claim}
	\begin{proof}[Proof of the Claim]
	For notational simplicity, let $\dot{\sigma}_0=\dot{\sigma}_{\mathrm{left}}$ and $\dot{\sigma}_1=\dot{\sigma}_{\mathrm{right}}$.
	Suppose otherwise, let $(p,s,t)\in Add(\omega,\lambda)\times Add(\omega,1)\times Add(\omega,1)$ force that $\dot{\sigma}_0 E \dot{\sigma}_1$ and $\neg \dot{\sigma}_i E r$ for any $r\in V^{Add(\omega,\lambda)}$, where $i=0,1$. Let $A\subseteq \lambda$ be countable such that $p\in Add(\omega,A)$ and $\dot{E}$ is an $Add(\omega,A)$-name. More precisely, the defining formula and the parameters of $E$ are $Add(\omega, A)$-names. Fix some $\gamma\in \lambda-A$. Consider the following automorphism $\pi$ on $Add(\omega,\lambda)\times Add(\omega,1)\times Add(\omega,1)$: $\pi(a,b,c)=(a^*,b^*,c^*)$ where
	\begin{itemize}
	\item $a\restriction \lambda-\{\gamma\}=a^*\restriction \lambda-\{\gamma\}$,
	\item $a^*(\gamma)=b$,
	\item $b^*=a(\gamma)$,
	\item $c=c^*$.
	\end{itemize}
	Hence, $\pi(p,s,t)\Vdash \pi(\dot{\sigma}_0) \pi(\dot{E}) \pi(\dot{\sigma}_1)$. By the definition of $\pi$, we know that $\pi(p,s,t)$ is compatible with $(p,s,t)$, $\pi(\dot{E})=\dot{E}$, $\pi(\dot{\sigma}_1)=\dot{\sigma}_1$ and $\pi(\dot{\sigma}_0)$ is an $Add(\omega, \lambda)$-name. Let $G\times g_0\times g_1\subseteq Add(\omega,\lambda)\times Add(\omega,1)\times Add(\omega,1)$ be generic containing both $(p,s,t)$ and $\pi(p,s,t)$. Then in the generic extension we have that $\sigma^* E \sigma_1$ where $\sigma^* = (\pi(\dot{\sigma}_0))^{G\times g_0 \times g_1}\in V[G]$, since $\pi(\dot{\sigma}_0)$ is an $Add(\omega, \lambda)$-name. However, as $(p,s,t)\in G\times g_0\times g_1$, we have that $\neg(\sigma^* E \sigma_1)$, which is a contradiction.
	\end{proof}
Let $V^*=V[G]$ where $G\subseteq Add(\omega, \lambda)$ is generic over $V$. Work in $V^*$.

Consider the following forcing $Q$: $p\in Q$ if and only if $p: 2^{\leq n}\to Add(\omega,1)$ for some $n\in \omega$ such that
\begin{itemize}
\item for each $s\in 2^{\leq n}$, there is some $k_s\in \omega$ such that $p(s)\Vdash \dot{\sigma}\restriction k_{s}=\tau_s$,
\item for $s\neq s'\in 2^m$ and $m\leq n$, $\tau_{s}\perp \tau_{s'}$, namely $\tau_s\cup \tau_{s'}$ is not a function.
\item for $s\sqsubseteq s'\in 2^{\leq n}$, $p(s')\leq p(s)$ and $\tau_{s'}\sqsupset \tau_{s}$, namely $\tau_s$ is a proper initial segment of $\tau_{s'}$.
\end{itemize}
The order of $Q$ is inclusion.
Since $Q$ is a non-trivial countable forcing, $Q$ is forcing-equivalent to $Add(\omega,1)$.
Let $T\subseteq Q$ be generic over $V^*$. In $V^*[T]$, an easy density argument shows that $T$ is a perfect subtree of $2^{<\omega}$ and the branches of $T$ are mutually generic Cohen reals over $V^*$. Consider $\{\sigma_b: b\in [T]\}$. By Claim \ref{claim:notEequiv}, we know that if $b\neq b'\in [T]$, then in $V^*[b,b']$, $\neg \sigma_b E \sigma_{b'}$. By Lemma \ref{lemma: 2step}, in $V^*[T]$, $\neg \sigma_b E \sigma_{b'}$. It remains to see that $\{\sigma_b: b\in [T]\}$ is a perfect set. This is the case since $\sigma_b = \bigcup_{s\sqsubseteq b} \tau_s$ for any $b\in [T]$.
\end{proof}

\begin{theorem}
Fix a regular cardinal $\kappa\geq \omega_2$. Over a model of CH, $\Vdash_{Add(\omega,\kappa)}$ every $\sigma$-projective equivalence relation either has $\leq \aleph_1$ or perfectly many equivalence classes.
\end{theorem}
\begin{proof}
Let $p\in Add(\omega,\kappa)$ and let $\dot{E}$ be a name for a thin $\sigma$-projective equivalence relation, namely, the $\sigma$-projective formula along with the parameters that define it. Since the $\sigma$-projective  formula and the parameters are essentially countable, there exists $A\subseteq \kappa$ such that $\omega_1\subseteq A$ and $|A|=\aleph_1$ such that $p\in Add(\omega, A)$ and $\dot{E}$ is an $Add(\omega , A)$-name. By Lemma \ref{lemma: 2step}, we know that $p\Vdash_{Add(\omega,A)} \dot{E}$ is thin. Indeed, if  some extension $q$ of $p$ forces that $\dot{E}$ is not thin, then Lemma \ref{lemma: 2step} implies that $\dot{E}$ is not thin in $V^{Add(\omega,\kappa)}$, which is impossible.
	\begin{claim}\label{claim: notadding}
	In $V^{Add(\omega,A)}$, $\Vdash_{Add(\omega,1)} $ ``for any $\sigma\in \mathbb{R}$, there is an $r\in V$ such that $\sigma E r$".
	\end{claim}
	\begin{proof}[Proof of the Claim]
	Otherwise, there is a $p\Vdash ``\exists \sigma\in \mathbb{R}$ such that $\neg (\sigma E r)$ for any $r\in V$". By the homogeneity of Cohen forcing, we have $\Vdash ``\exists \sigma\in \mathbb{R}$ such that $\neg (\sigma E r)$ for any $r\in V$". By Lemme \ref{lemma: mutual}, we know that $\Vdash_{Add(\omega,1)} $ ``there exist perfectly many $E$-classes". By Lemma \ref{lemma: 2step}, $\Vdash_{Add(\omega,\kappa)}$ ``there exist perfectly many $E$-classes", contradicting the assumption on the thinness of $E$.
	\end{proof}

Consequently, Claim \ref{claim: notadding} implies that in $V^{Add(\omega, A)}$, $\Vdash_{Add(\omega, \kappa-A)}$ ``for any $\sigma\in \mathbb{R}$, there is an $r\in V$ such that $\sigma E r$". Since $V^{Add(\omega, A)}$ is a model of CH, we have that the number of $E$-classes in $V^{Add(\omega, \kappa)}$ is $\leq \aleph_1$.
\end{proof}

We conclude the paper with some observations that the Second-order Absolute Morley in general is strictly stronger than its non-absolute version. 

\emph{Second-order Morley}: If $T$ is a second-order theory in a countable signature, then either $T$ has at most $\aleph_1$ isomorphism classes of countable models, or there are $2^{\aleph_0}$ many non-isomorphic countable models of $T$.

If $2^{\aleph_0} = \aleph_2$, Second-order Morley trivially holds. Although Second-order Absolute Morley holds in the Cohen model, Second-order Morley does not imply Second-order Absolute Morley in general. 

\begin{lemma}\label{lemma: dich}
Second-order Absolute Morley implies that for any light-face projective set $A$, either $A$ has size $\leq \aleph_1$ or $A$ includes a perfect set.
\end{lemma}

\begin{proof}
The proof is similar to that of Lemma 3.1 in \cite{Eagle2023}, adapted to the ``absolute'' scenario. Therefore, we will only sketch the proof. The second-order theory $T$ considered here is second-order Peano Arithmetic with an additional unary predicate $R$. Let $\mathcal{L}$ be the natural language for this. Intuitively speaking, this predicate is coding a real that belongs to $A$. Namely, our model will be of the form of $\mathcal{A}=(\omega, + , \cdot, <, 0,1, R^{\mathcal{A}})$ with $R^{\mathcal{A}}\in A$. There is a natural translation of a projective formula $\psi$ to a second-order $\mathcal{L}$-formula $\psi^{\mathcal{L}}$ that is truth-preserving (see \cite[8B.15]{Mos80} for more information).
In particular, if $\varphi(x)$ is a projective definition of $A$, then the requirement that $R^\mathcal{A} \in A$ can be expressed as $\mathcal{A}\models ``\exists Y ((\forall n ( n\in Y \leftrightarrow n\in R)) \wedge \varphi^\mathcal{L}(Y))$". As a result, the above procedure describes a second-order theory $T$ in the language $\mathcal{L}$ such that any $\mathcal{A}\models T$,  $R^\mathcal{A}\in A$. Note that for any $\mathcal{A}_0, \mathcal{A}_1\in Mod_T$ (namely, the collection of models for the theory $T$), $\mathcal{A}_0 \cong \mathcal{A}_1$ if and only if $R^{\mathcal{A}_0}=R^{\mathcal{A}_1}$. Also the natural product topology on $Mod_T$ is generated by $U_\sigma=\{\mathcal{A}\in Mod_T: \sigma\sqsubset R^{\mathcal{A}}\}$ for $\sigma\in 2^{<\omega}$.

Assume that $|A|>\aleph_1$.
Apply Second-order Absolute Morley. By the assumption that $|A|>\aleph_1$, $Mod_T$ must contain more than $\aleph_1$ many non-isomorphic models. In particular, there is a continuous injection $\pi': 2^\omega\to Mod_T$ such that the images of $\pi'$ are pairwise non-isomorphic. Let $\pi: 2^\omega\to A$ be defined on $2^\omega$ such that $\pi(\sigma)=R^{\pi'(\sigma)}$. We check that $\pi$ is a continuous injection. The fact that $\pi$ is an injection follows from the fact that for $\sigma_0\neq \sigma_1$, $\pi'(\sigma_0)\not\simeq \pi'(\sigma_1)$, which implies $R^{\pi'(\sigma_0)}\neq R^{\pi'(\sigma_1)}$. For any basic open set $[\sigma]_A =_{def} \{g\in A: \sigma\sqsubset g\}$ where $\sigma\in 2^{<\omega}$ and $f\in \pi^{-1}([\sigma]_A)$, since $\pi'$ is continuous and $f\in (\pi')^{-1}(U_\sigma)$, there is some $\tau\in 2^{<\omega}$ such that $f\in [\tau]=_{def}\{g\in 2^\omega: \tau\sqsubset g\}\subset (\tau')^{-1}(U_\sigma)$. For any $h\in [\tau]$, $\pi(h)=R^{\pi'(f)}\sqsupset \sigma$.
 As a result, $f\in [\tau] \subset \pi^{-1}([\sigma]_A)$. Hence $\pi$ is continuous.
	
\end{proof}

\begin{remark}
The proof above still works if $A$ is bold-face. In this case we just need to add a unary predicate for each of the real parameters in the definition of $A$.
\end{remark}

\begin{thm}\label{thm:Sigma1reflect}
    There is a model of $\ZFC$ in which $2^{\aleph_0} = \aleph_2$ and in which there is a projective equivalence relation on $\R$ which has $2^{\aleph_0}$-many equivalence classes but does not have a perfect set of equivalence classes.
\end{thm}

\begin{proof}
It is well-known that it is consistent that there is a light-face projective well order of the reals and $2^{\aleph_0}=\aleph_2$ (see \cite{CaicedoSchindler}). 
We use a projective well order of the reals to produce a projective set witnessing the failure of Lemma \ref{lemma: dich}. More precisely, we will construct a projective \emph{Bernstein set} $B$, i.e. such that neither $B$ nor $\R\setminus B$ includes a perfect set. In detail, let $\preceq$ be the projective well order. Recursively coding pairs of reals as reals, $\preceq$ induces a projective well order on pairs of reals that we will also call $\preceq$. For each real, we can ask if it codes a perfect subtree of $2^{<\omega}$ by some recursive bijection between $\omega$ and $2^{<\omega}$. Define $f: \R \to \R^2$ such that $f(x) = \langle y_x, z_x\rangle$ if:
    \begin{enumerate}[label=\arabic*)]
        \item\label{cond1} for any $x' \prec x$, $y_x$ and $z_x$ are not equal to either $y_{x'}$ or $z_{x'}$,
        \item\label{cond2} if $x$ codes a perfect subtree $T_x$ of $2^{<\omega}$, then $y_x$ and $z_x$ are different branches through $T_x$, the tree determined by $x$,
        \item $\langle y_x, z_x\rangle$ is the $\preceq$-least pair satisfying \ref{cond1} and \ref{cond2}.
    \end{enumerate}

    Since $\preceq$ is projective, so is $f$. Let
    \[B = \st{y}{\exists x \exists z(f(x) = \langle y, z\rangle)}.\]
    Then $y=y_x \in B$ implies $z_x \notin B$.

    Then $B$ is projective since $f$ is. However we claim that neither $B$ nor $\R\setminus B$ includes a perfect set, for if such a perfect set $P$ were realized as a perfect $T_x$, then $y_x$ would be in $B$ and $z_x$ would be in $\R\setminus B$, yet both are in $P$.
\end{proof}

\section{Acknowledgment}
We thank the referee for their comments, suggestions and corrections that greatly improve the quality of this paper.

\bibliographystyle{amsplain}
\bibliography{ms}

\providecommand{\bysame}{\leavevmode\hbox to3em{\hrulefill}\thinspace}
\providecommand{\MR}{\relax\ifhmode\unskip\space\fi MR }
\providecommand{\MRhref}[2]{%
  \href{http://www.ams.org/mathscinet-getitem?mr=#1}{#2}
}
\providecommand{\href}[2]{#2}
\begin{thebibliography}{10}

\bibitem{Aguilera}
J.~Aguilera, \emph{$\sigma$-projective {D}eterminacy}, preprint,
  \url{https://www.dropbox.com/s/g853hj5ern1kkyb/SPReversal.pdf?dl=0}.

\bibitem{AMS21}
J.~P. Aguilera, S.~Müller, and P.~Schlicht, \emph{Long games and
  $\sigma$-projective sets}, Annals of Pure and Applied Logic \textbf{172}
  (2021), 102939.

\bibitem{Burgess}
J.~Burgess, \emph{Infinitary languages and descriptive set theory}, Ph.D.
  thesis, University of California at Berkeley, Berkeley, California, 1974.

\bibitem{CaicedoSchindler}
Andr\'{e}s~Eduardo Caicedo and Ralf Schindler, \emph{Projective well-orderings
  of the reals}, Arch. Math. Logic \textbf{45} (2006), no.~7, 783--793.
  \MR{2266903}

\bibitem{Eagle2023}
Christopher~J. Eagle, Clovis Hamel, Sandra M\"{u}ller, and Franklin~D. Tall,
  \emph{An undecidable extension of {M}orley's theorem on the number of
  countable models}, Ann. Pure Appl. Logic \textbf{174} (2023), no.~9, Paper
  No. 103317, 25. \MR{4617948}

\bibitem{FM95}
M.~Foreman and M.~Magidor, \emph{{Large cardinals and definable counterexamples
  to the continuum hypothesis}}, Annals of Pure and Applied Logic \textbf{76}
  (1995), no.~1, 47--97.

\bibitem{Gao}
S.~Gao, \emph{Invariant descriptive set theory}, Chapman and Hall/CRC, New
  York, 2008.

\bibitem{Kechris}
A.~{ }~S. Kechris, \emph{Classical descriptive set theory}, Springer-Verlag,
  New York, 1995.

\bibitem{Marker}
David Marker, \emph{Lectures on infinitary model theory}, Lecture Notes in
  Logic, vol.~46, Association for Symbolic Logic, Chicago, IL; Cambridge
  University Press, Cambridge, 2016. \MR{3558585}

\bibitem{Morley1970}
M.~Morley, \emph{The number of countable models}, J. Symb. Logic \textbf{35}
  (1970), 14--18.

\bibitem{Mos80}
Y.N. Moschovakis, \emph{Descriptive set theory}, Mathematical Surveys and
  Monographs, vol. 155, American Mathematical Society, Providence, R.I., 2009.

\bibitem{Vaught}
R.~Vaught, \emph{Denumerable models of complete theories}, Proc. Sympos.
  Foundations of Mathematics, Infinitistic Methods (Warsaw), Pergamon Press,
  1961, pp.~303--321.

\end{thebibliography}

\end{document}